\documentclass[a4paper,12pt]{amsart}
\usepackage[T1]{fontenc}
\usepackage[centering,margin=2.3cm]{geometry}
\usepackage[latin1]{inputenc}
\usepackage{amsthm, amsmath}   
\usepackage{latexsym,amssymb}
\usepackage{algorithmic}
\usepackage{amssymb}
\usepackage{amsaddr}
\usepackage{graphicx,subfigure}
\usepackage{times}
\usepackage{enumerate}
\usepackage[usenames,dvipsnames]{pstricks}
\usepackage{epsfig}
\usepackage{pst-node}
\usepackage{pst-grad} 
\usepackage{pst-plot} 
\usepackage{pst-3dplot}
\usepackage{color}
\usepackage{cite}
\newcommand{\upa}{\ensuremath{\uparrow\hspace{-1.5pt}}}
\newcommand{\downa}{\ensuremath{\downarrow\hspace{-1.5pt}}}

\newcommand{\N}{\mathbb{N}}

\newcommand{\Z}{\mathbb{Z}}

\newcommand{\End}{\mathrm{End}}
\newcommand{\Min}{\mathrm{Min}}

\newcommand{\lS}{\leq_{M}}

\theoremstyle{plain}
\newtheorem{theorem}{Theorem}[section]
\newtheorem{lemma}[theorem]{Lemma}
\newtheorem{corollary}[theorem]{Corollary}
\newtheorem{proposition}[theorem]{Proposition}

\newtheorem{observation}[theorem]{Observation}

\theoremstyle{definition}

\newtheorem{conjecture}[theorem]{Conjecture}

\newtheorem{question}[theorem]{Question}

\newtheorem{case}{Case}
\newtheorem{subcase}{Case}[case]

\title[Cayley posets]{Cayley posets}
\author[I. Garc\'{i}a-Marco]{Ignacio Garc\'{i}a-Marco}
\address{Facultad de Ciencias, Universidad de La Laguna, La Laguna, Spain}
\author[K. Knauer]{Kolja Knauer}
\address{Aix Marseille Univ, Universit\'e de Toulon, CNRS, LIS, Marseille, France\\Departament de Matem\`atiques i Inform\`atica,
Universitat de Barcelona (UB), Barcelona, Spain}
\author[G. Mercui-Voyant]{Guillaume Mercui-Voyant}
\address{\'Ecole Centrale Marseille, Technopole de Ch\^ateau-Gombert, Marseille, France}



\begin{document}

\begin{abstract}
We introduce Cayley posets as posets arising naturally from pairs $S<T$ of semigroups, much in the same way that Cayley graph arises from a (semi)group and a subset. We show that Cayley posets are a common generalization of several known classes of posets, e.g. posets of numerical semigroups (with torsion) and more generally affine semigroups. Furthermore, we give Sabidussi-type characterizations for Cayley posets and for several subclasses in terms of their endomorphism monoid. We show that large classes of posets are Cayley posets, e.g., series-parallel posets and (generalizations of) join-semilattices, but also provide examples of posets which cannot be represented this way. Finally, we characterize (locally finite, with a finite number of atoms) auto-equivalent posets - a class that generalizes a recently introduced notion for numerical semigroups - as those posets coming from a finitely generated submonoid of an abelian group.
\end{abstract}

\maketitle

\section{Introduction}\label{introduction}

Cayley graphs of groups are a classical topic in algebraic graph theory. They play a prominent role in (books devoted to) the area, see e.g.~\cite{God-01}. A particular and central result of the theory, due to Sabidussi~\cite{Sab-58} characterizes Cayley graphs of groups via the action of their automorphism group. Cayley graphs of monoids and semigroups have been less studied, but still there is a considerable amount of work, see e.g.~\cite{Kna-19} for a book and the references therein. In semigroups, analogues of the above result of Sabidussi remain open. Characterizations of Cayley graphs of certain classes of semigroups have been subject to some research effort, see~\cite{zbMATH05610940,Kel-06,zbMATH06864652,zbMATH06740692,zbMATH06948232,zbMATH06613956,zbMATH06184562,zbMATH06147894,zbMATH06029728,zbMATH06120589,zbMATH06093204,zbMATH05973394,zbMATH05886836,zbMATH06139402,zbMATH05701935,zbMATH05812671}. Conversely, also characterizations of semigroups admitting Cayley graphs with certain (mostly topological) properties have been investigated extensively~\cite{zbMATH05811972,zbMATH05647985,zbMATH06603377,Kel-02,Kel-03,Sol-06,Sol-11,Zha-08,Kna-10,Kna-16}. With respect to groups there is a well-known book concerning the topic~\cite{Whi-01}.

In the present paper we pursue questions of this type in a setting which naturally excludes groups from the picture. Namely, we study semigroups whose Cayley graph yields a partially ordered set (poset). We call the resulting posets \emph{Cayley posets}. Such objects arise also naturally from considering (relative) Green's relations on semigroups.
An important class of Cayley posets are numerical semigroups (see the books~\cite{Ros-09,Ass-16,Ram-05}), and more generally numerical semigroups with torsion (see~\cite{Gar-18,Lop-14}) and affine semigroups (see the books ~\cite{CLS,RGSlibro,Mil-05, Villalibro}). Further examples include families defined in order to construct diamond free posets in~\cite{Cza-15}.

Let us define the setting of this article in the slightly more general  language of acts, see~\cite{Kil-00}. Let $(X,S)$ be a \emph{semigroup (right) act}, i.e., $X$ is a set and $(S,\cdot)$ a semigroup with an operation such that $xs\in X$ and $(xs)s'=x(s\cdot s')$ for all $x\in X$ and $s,s'\in S$. We define the binary relation $\leq_S$ on $X$ via $x\leq _S y :\iff \exists s\in S: xs=y$.\footnote{Note that if $X$ is a (semi)group, $S\subseteq X$ is just a subset, and the $\cdot$ is right-multiplication, then the resulting relation coincides with the Cayley graph of $X$ with respect to $S$.}

In this paper we study acts such that $\leq_S$ is an order-relation. In this case, we denote the poset given by $\leq_S$ on $X$ as $P(X,S)$ and call it the \emph{Cayley poset} of $(X,S)$. Another way of seeing the objects we study is from the point of view of \emph{Green's relations}. The semigroup (acts) we consider are such that the R-classes of their Green's relations are trivial. Note that in the study of Green's relation the order based on the R-relation is often defined as $x\leq _S y :\iff \exists s\in S: sx=y$, but we stick to our convention motivated from directed (right) Cayley graphs. See standard semigroup books such as~\cite{Cli-61,How-76} for more details on Green's relation.

Let us give a more detailed overview of the objects and results of the present paper. First, we characterize the acts that generate posets as those that are \emph{s-unital} and \emph{acyclic} (Proposition~\ref{prop:sunital}). Then we prove that in fact every poset is the Cayley poset of a monoid act (Theorem~\ref{thm:everybodyisanact}). We thus consider natural special cases of Cayley posets of acts (always assuming them to be s-unital and acyclic). If $X=T$ is a semigroup such that $S<T$ is a subsemigroup and the act consists of right-multiplication, we say that the Cayley poset $P(T,S)$ is a \emph{semigroup poset}. If furthermore $S=T$ we say that $P(T,S)=P(S,S)$ is \emph{full}. If $T$ and $S$ are monoids $N,M$, and $M < N$ is a submonoid we say that $P(T,S)=P(N,M)$ is a \emph{monoid poset}. See Figure~\ref{fig:N} for an example.

\begin{figure}[h]
\centering
\includegraphics[width=.5\textwidth]{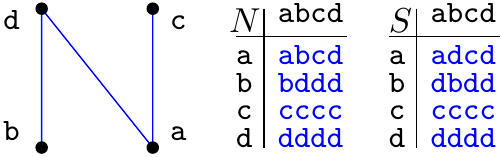}
\caption{The N-poset is a monoid poset via $P(N,\{a,c,d\})$ and is full as witnessed by $P(S,S)$.}
\label{fig:N}
\end{figure}


With these definitions we can describe well-known classes of posets as special cases of the above mentioned. An important class of Cayley posets are affine semigroups, that is, $P(N,M)$ where $N = \Z^n$ with the usual addition, and $M$ is a submonoid of $\Z^n$ without invertible elements different from $0$. These posets have been studied due to their relation with (binomial) lattice ideals  (see, for example, \cite{BCMP,Villalibro,Mil-05, Mor-05,Kat-10,Eis-96}). As a prominent subfamily of affine semigroups, one has numerical semigroups, which can be seen as $P(M,M)$ where $M$ is a submonoid of $\N$, see Figure~\ref {fig:numsem} for an example and~\cite{Ros-09,Ass-16,Ram-05} for literature on the topic. Recently, numerical semigroups with torsion were studied in~\cite{Gar-18,Lop-14}, where $N=\mathbb{N}\times T$ for a finite Abelian group $T$. Further examples of monoid posets include families defined in order to construct diamond free posets in~\cite{Cza-15}. Here for a group $G$ and a set of generators $H$ the corresponding Cayley poset is defined as $P(G\times\mathbb{Z},\langle H\times\{1\}\rangle\cup\{(e,0)\})$.

\begin{figure}[h]
\centering
\includegraphics[height=.4\textwidth]{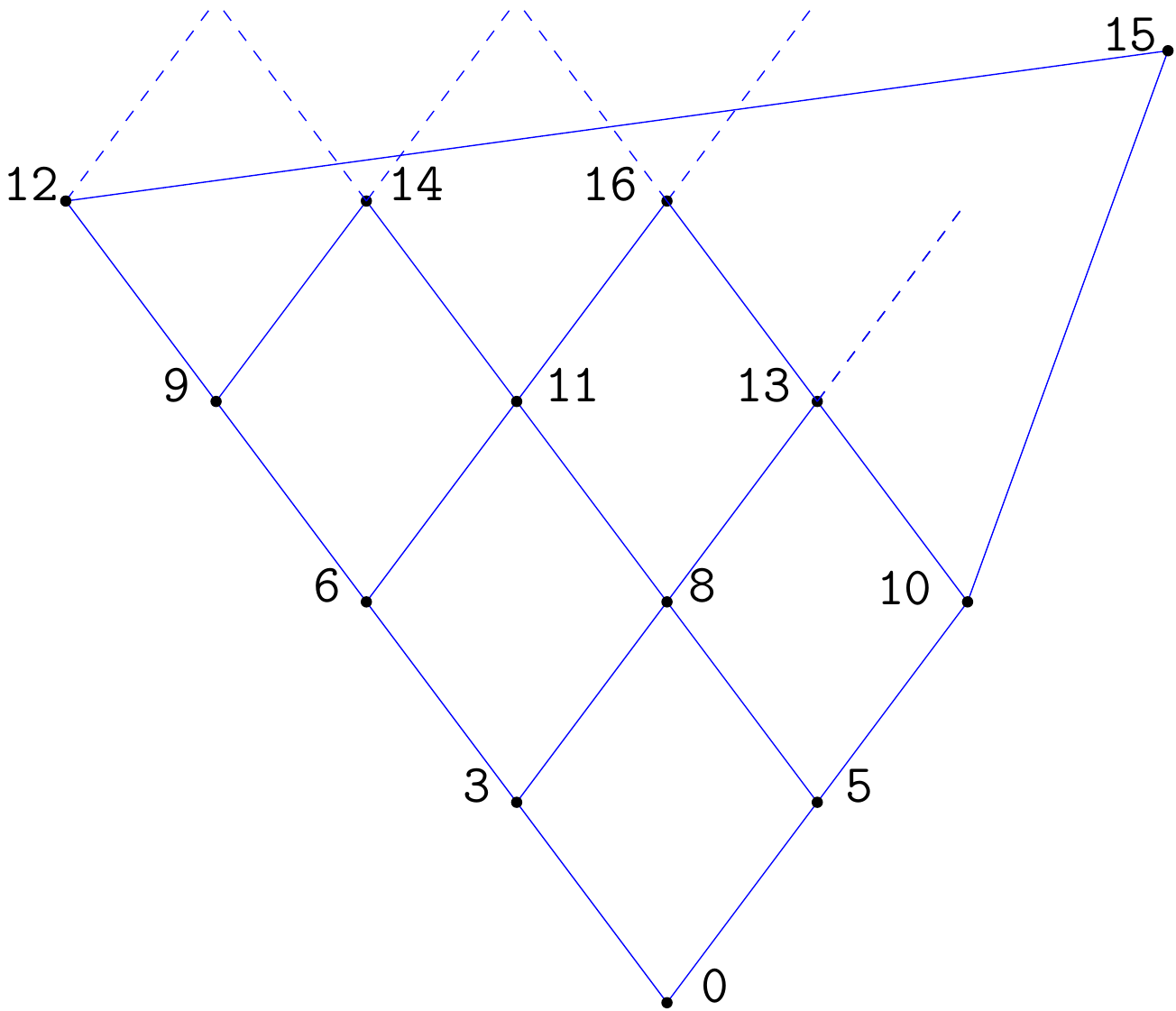}
\caption{The numerical semigroup $P(M,M)$, where $M$ is the submonoid of $\N$ generated by $3$ and $5$.}
\label{fig:numsem}
\end{figure}

Here, we present Sabidussi-type characterization results for semigroup posets (Theorem~\ref{thm:characterizationsemi}), monoid posets (Theorem~\ref{thm:characterizationmono}), full posets (Corollary~\ref{cor:characterizationfull}) and full monoid posets (Theorem~\ref{thm:characterizationfullmono}), i.e., these are characterizations in terms of the endomorphism monoid of the poset much like Sabidussi's classical characterization of Cayley graphs of groups~\cite{Sab-58}.

There are natural inclusions among the considered classes, (as mentioned above Cayley posets of acts coincide with all posets). In Sections \ref{sec:firstseparations},  \ref{sec:construction} and \ref{sec:examples} we give some properties of different types of Cayley posets and show that all these inclusions are strict, by giving posets in the respective difference sets. See the diagram in Figure~\ref{fig:diagram} for an illustration.

\begin{figure}[h]
\centering
\includegraphics[width=.8\textwidth]{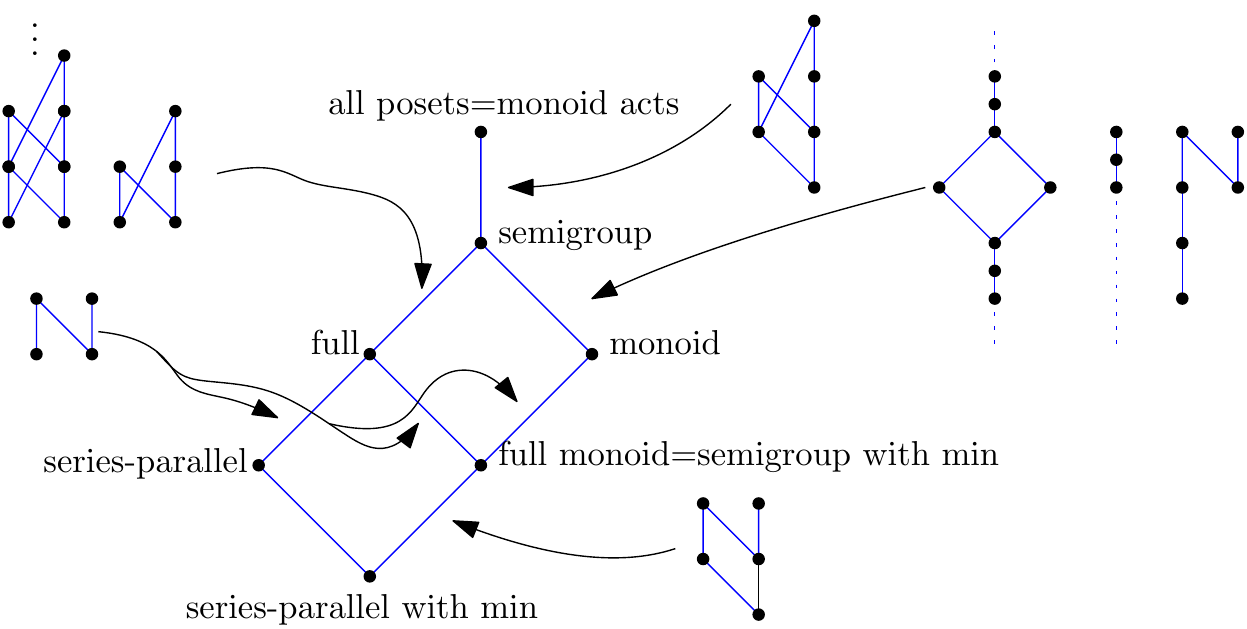}
\caption{Inclusions among considered classes and posets witnessing strictness of containment.}
\label{fig:diagram}
\end{figure}

Partially thanks to the characterization theorems we are able to show that semigroup posets are closed under natural poset operations such as addition of maxima or minima, Cartesian products, retracts, series and parallel composition, and certain blow-up operations. This way we construct large families of semigroup posets in Section~\ref{sec:construction}. A particular result  of this section is that series-parallel posets are full.

In~\cite[Section 5]{Cha-15} the authors characterize posets coming from finitely generated submonoids $M$ of $\Z^m$. In this result, there is a rather technical condition saying that a certain group has to be saturated. Section~\ref{sec:autoequivalent} is devoted to a generalization of this result. In Theorem \ref{equivasemigrupo} we prove that posets  coming from finitely generated submonoids of an abelian group are characterized by dropping this technical assumption. 


%

We conclude by resuming several questions that come up through  the paper  in Section~\ref{sec:conclude}.

\section{Characterizations}\label{sec:characterization}
In this section we present several characterization results for classes of Cayley posets. First, let us describe under which conditions on a semigroup act $(X,S)$ the  relation $\leq _S$ is an order relation.
\begin{proposition}\label{prop:sunital}
 Let $(X,S)$ be a semigroup (right) act. The relation $\leq _S$ is an order relation if and only if $(X,S)$ satisfies
 \begin{itemize}
 \item $\forall x\in X~\exists s\in S: xs=x$ \hfill (\emph{s-unital})
 \item $\forall x\in X~\forall s,s'\in S:x=xss'\Longrightarrow x=xs$ \hfill (\emph{acyclic})
\end{itemize}
\end{proposition}
\begin{proof}
 Clearly, $x\leq_S y$ and $y\leq_S z$ implies $xs=y$ and $ys'=z$ for some $s,s'\in S$. Thus, $z=(xs)s'=x(s\cdot s')$ and $x\leq_S z$, i.e., $\leq_S$ is transitive.
 
 Obviously, $\leq_S$ is reflexive if and only if the act is s-unital.
 
 Let us now see that $\leq_S$ is anti-symmetric if and only if the act is acyclic. We can do the following equivalent transformations starting with acyclicity, i.e., $x\leq_S y$ and $y\leq_S x\implies x=y$, which is equivalent to $xs=y$ and $ys'=x\implies x=y$, which in turn is equivalent to $xss'=x\implies xs=x$.
\end{proof}

Next, we show that indeed every poset is the Cayley poset of a monoid act.
\begin{theorem}\label{thm:everybodyisanact}
 For every poset $P$ there is a monoid act $(X,M)$ such that $P\cong P(X,M)$.
\end{theorem}

\begin{proof}

Denote the ground set of the poset $P$ by $X$ and let $M = \lbrace \varphi : P \rightarrow P \mid \forall x \in P, x \leq \varphi(x) \rbrace$ with the operation $\varphi\varphi':=\varphi'\circ\varphi$. We prove that the mapping $x\cdot \phi:=\phi(x)$ is a monoid act such that $P\cong P(X,M)$.

Let $\varphi , \varphi' \in M$ then $\forall x \in P$ we have $\varphi\varphi' (x)=\varphi' \circ \varphi (x)= \varphi'(\varphi(x)) \geq \varphi(x) \geq x$ thus $\varphi\varphi' \in M$ and $ M$ is a semigroup. Moreover $id \in M$ thus $M$ is a monoid. Moreover, we clearly have $x\cdot \varphi\in X$ and $(x\cdot \varphi)\varphi'=x\cdot(\varphi\varphi')$ for all $x\in X$ and $\varphi,\varphi'\in M$. Thus $(X,M)$ is a monoid act. Now, let us show that the relation $x\leq_{\circ} y \Longleftrightarrow \exists_{\varphi \in M}: y = \varphi(x)$ is s-unital and acyclic. First, $ \forall x \in M , x \cdot id = x$ which shows that $\leq_{\circ}$ is s-unital. Second, for $x\in P$ and $\varphi,\varphi'\in M$ $ x\cdot \varphi \varphi' = x \Rightarrow \varphi(x) \leq \varphi'(\varphi(x)) =x$ and $x\leq \varphi(x)$ then $ \varphi(x) = x $ thus $\leq_{\circ}$ is acyclic. Now we will show that $\forall x,y \in P$ we have $x \leq y \Leftrightarrow x\leq_{\circ} y$. 

"$\Rightarrow$" : if $x \leq y$ then define $\varphi_{x,y}(z) := y$ if $z= x$ and $z$ otherwise. Clearly, $\varphi_{x,y} \in M$ and $ \varphi_{x,y} (x) =y$ thus $x  \leq_{\circ} y$.

"$\Leftarrow$" : if $x  \leq_{\circ} y$ then $y = x\cdot \varphi$ for some $\varphi\in M$. Thus, $y=\varphi(x)$, which by the definition of $M$ implies $x \leq y$.

This concludes the proof.

\end{proof}

For a poset $P$, the monoid of order endomorphisms of $P$ with the composition will be denoted by $\End(P)$. We recall that an \emph{upset} (resp. a \emph{downset}) in a poset $P$ is an upward (resp. downward) closed  subset $F\subseteq P$. The \emph{principal upset and downset of $x$} will be denoted by  $\upa x:=\{y\in P\mid x\leq y\}$ and $\downa x:=\{y\in P\mid x\geq y\}$, respectively\footnote{Upsets and downsets are sometimes called \emph{filters} and \emph{ideals}, respectively.} . 

For a semigroup $T$, we set $L(T):=\{\varphi_t \mid t\in T \}$, where $\varphi_t:T \to T$ is defined as $\varphi_t(x) = tx$, i.e., $\varphi_t$ is the left-multiplication by $t$. 
Since $\varphi_t \circ \varphi_{t'} = \varphi_{t \cdot t'}$, it follows that $L(T)$ is a semigroup with the composition and that $t \mapsto \varphi_t$ is a semigroup epimorphism from $(T,\cdot)$ to $(L(T), \circ)$. Moreover, $L(T)$ is a monoid whenever $T$ has a left identity.

The following lemma will be useful for the upcoming characterizations of different types of Cayley posets. 
\begin{lemma}\label{lem:basic}
Let $S<T$ semigroups and $P \cong P(T,S)$ a semigroup poset. Then, $L(T)$ is a subsemigroup of ${\rm End}(P)$. Moreover, $(L(T),\circ)$ and $(T,\cdot)$ are isomorphic semigroups.  
\end{lemma}

\begin{proof}
Let $\varphi_t \in L(T)$ and consider $t',t'' \in T$ such that $t' \leq t''$. Taking $s\in S$ such that $t's = t''$, we have that $\varphi_t(t')\leq \varphi_t(t')s=tt's=tt''=\varphi_t(t'')$. Thus, $L(T)\subseteq \End(P)$ and, hence, $L(T)$ is a subsemigroup of $\End(P)$.
 Finally, let us show that $t \mapsto \varphi_t$ is an injection. Assume that $\varphi_{t} = \varphi_{t'} $. Since $P(T,S)$ is a semigroup poset it follows that $\upa t' = \varphi_{t'}(S) = \varphi_{t}(S) = \upa t$, which implies that $t' = t$.
\end{proof}

We now proceed to characterize semigroup posets.

\begin{theorem}\label{thm:characterizationsemi}
Let $P$ be a poset. There are semigroups $S<T$ such that $P \cong P(T,S)$ if and only if there is a subsemigroup $L<\End(P)$ and an upset $F\subseteq P$ such that for every element $x\in P$ there is a unique $\varphi_x\in L$ such that $\varphi_x(F)=\upa x$. Moreover, in this case $S=F$ and $L=L(T)$.
\end{theorem}
\begin{proof}

"$\Rightarrow$": Let $P\cong P(T,S)$. By Lemma~\ref{lem:basic} we have $T\cong L(T) < \End(P)$. For every $x\in T$ the mapping $\varphi_x$  exists and is unique because $T\cong L(T)$. By the definition of $P(T,S)$, $S$ is clearly  an upset of $P$ and $\varphi_x(S) = \upa x$.

"$\Leftarrow$": Let $T$ be the ground set of $P$. For every $x,y \in T$, we define the operation $xy := \varphi_x(y)$. Before proving that this operation is associative we are going to show that  $\varphi_x \circ \varphi_y = \varphi_{\varphi_x(y)}$. Indeed, $\varphi_x \circ \varphi_y = \varphi_w$ for some $w \in P$ because $L$ is a semigroup. Moreover,  $\upa w = \varphi_w(F) = \varphi_x \circ \varphi_y(F) = \varphi_x(\upa y)$; and $\varphi_x(\upa y) \subseteq\, \upa \varphi_x(y)$ because $\varphi_x$ is an endomorphism of $P$. This proves $w \geq \varphi_x(y)$. Taking $s \in F$ so that $\varphi_y(s) = y$, we get that $\varphi_x(y) = \varphi_x \circ \varphi_y(s) = \varphi_w(s) \in \upa w$, i.e., $ \varphi_x(y) \geq w$. Thus, $w = \varphi_x(y)$ and we have that: $$(xy)z = \varphi_{xy}(z) = \varphi_{\varphi_x(y)}(z)= \varphi_x \circ \varphi_y(z)= \varphi_x(\varphi_y(z))= \varphi_x(yz) = x(yz).$$ Set $S := F$  and let us see that $S$ is a subsemigroup of $T$. Take $x,y \in S$, since $S$ is an upset we have that $xy=\varphi_x(y) \in\, \upa x\subseteq S$. Finally, let $x,y\in T$ such that $x \leqslant y$. Thus, $y \in \varphi_x(S)$, i.e., $y = xs$ for some $s\in S$. Conversely, if $y \in \varphi_x(S)$, then $y = \varphi_x(s) = x s$ for some $s \in S$ and, hence, $y \geq x$. This proves that $P\cong P(T,S)$.
\end{proof}

Next, we quickly deduce a characterization of full posets.
\begin{corollary}\label{cor:characterizationfull}
 A poset $P$ is a full semigroup poset if and only if there is a subsemigroup $T<\End(P)$ such that for every element $x\in P$ there is a unique $\varphi_x\in T$ such that $\varphi_x(P)=\upa x$.
\end{corollary}
\begin{proof}
By Theorem~\ref{thm:characterizationsemi} we have that in a semigroup poset $P(T,S)$ the upset $F$ corresponds to $S$. Thus, the fact that $S=T$ corresponds to $F=P$. This yields the result.
%
%
\end{proof}

The following theorem gives a characterization of monoid posets.
\begin{theorem}\label{thm:characterizationmono}
 Let $P$ be a poset. There are monoids $M<N$ such that $P \cong P(N,M)$ if and only if there is a submonoid $L<\End(P)$ and a principal upset $\upa e\subseteq P$ such that for every element $x\in P$ there is a unique $\varphi_x\in L$ such that $\varphi_x(\upa e)=\upa x$. Moreover, in this case $\upa e=M$ and $L=L(N)\cong N$.
\end{theorem}

\begin{proof}

"$\Rightarrow$":  We first observe that $M =\, \upa e$, where $e \in N$ is the identity element.  By Theorem \ref{thm:characterizationsemi} we have that 
$N \cong L(N)$ and $L(N)$
is a subsemigroup of ${\rm End}(P)$. Since $\varphi_e = id$, we conclude that $R$ is a submonoid of ${\rm End}(P)$.

"$\Leftarrow$": We take $N = L$ and $M = \upa e$, by Theorem \ref{thm:characterizationsemi} if suffices to prove that $N$ is a monoid and that $M$ is a submonoid.
Since $L$ is a submonoid of ${\rm End}(P)$, there exists $x \in P$ such that $\varphi_x = id$. Hence $\varphi_x(e) = id(e)$, which implies that $x = e \in M$ and $e$ is the identity. Thus, $N$ is a monoid and $M$ is a submonoid of $N$.
\end{proof}

We do not know, whether only requiring $L$ to be a subsemigroup in Theorem~\ref{thm:characterizationmono} yields a strictly larger class of semigroup posets.
To finish this section, we provide a slightly different type of characterization for full monoid posets. Before proving the characterization we introduce an easy lemma that will be useful in the forthcoming.

%
%
%

\begin{lemma}\label{lemafacil} Let $P \cong P(T,S)$ and let $m \in T$. If $m = mm$, then $x = mx$ for all $x \in\, \upa m$.
\end{lemma}
\begin{proof}If $x \in \upa m$, then there exists an $y \in S$ such that $x = my$; then $x = my = mmy= mx$.
\end{proof}

\begin{theorem}\label{thm:characterizationfullmono}
 A poset $P$ is a full monoid poset if and only if $P$ is a semigroup poset with global minimum.
\end{theorem}

\begin{proof}
"$\Rightarrow$": Assume that $P \cong P(M,M)$ and denote by $e \in M$ its identity element. For all $x \in M$ we have that $ex = x$, thus $x \geqslant e$. We conclude
that $e$ is the minimum of $P$.

"$\Leftarrow$"  Assume that $P \cong P(T,S)$ and denote by $m$ the minimum of $P$. We have that $T = \ \upa m = \{ms \, \vert \, s \in S\}$. Then there exists an $s \in S$ such that $m = m s$. Moreover, since $m$ is the minimum, it follows that $m \leq m m \leq m s = m$ and we conclude that $mm = m$. By Lemma \ref{lemafacil} and the fact that $T =\ \upa m$, it follows that $m x = x$ for all $x \in T$. Thus $T =\ \upa m = \{ms \, \vert \, s \in S\} = S$. To finish the proof it suffices to verify that $x m = x$ for all $x \in T$, but this follows easily because $\upa x = \{x s \, \vert \, s \in S\}$ and $m$ is the minimum of $S$. Then we conclude that $S = T$ is a monoid with identity $m$ and $P$ is a full monoid poset.
\end{proof}

\section{First separations between classes of posets}\label{sec:firstseparations}

Figure \ref{fig:diagram} shows different classes of posets partially ordered by inclusion and examples separating the classes. 
This section is devoted to partially justify  this figure, the rest of this proof will be done in Sections \ref{sec:construction} and \ref{sec:examples}. A first easy observation is that
every full monoid poset has a global minimum which corresponds to the neutral element (see Theorem \ref{thm:characterizationfullmono}). Hence the N-poset (which is a full semigroup poset and a monoid poset, see Figure \ref{fig:N}) is not a full monoid. Moreover, the N-poset is not a series-parallel poset. Indeed, it is exactly the forbidden induced subposet for the elements of this class, see~\cite{Val-82}.

The main results of this section are examples of:
\begin{itemize}
\item[(a)] monoid posets which are not full semigroup posets (Theorem \ref{th:monoidnotfull} for an infinite poset and Corollary \ref{co:N4semignotfull} for a finite one),
\item[(b)] full semigroup posets which are not monoid posets (Theorem \ref{th:fullnotmonoid}), and
\item[(c)] posets which are not semigroup posets (Theorem \ref{th:notCayley}).
\end{itemize}
From (a) and (b) we have that there is not containment between the classes of monoid posets and full semigroup posets. 
From (c) we derive that semigroup posets form a strict subfamily of the class of all posets.

The infinite poset mentioned in (a) consists of the natural numbers $(\mathbb{N},\leq)$ ordered by $a < b$ if and only if $b - a \geq 2$ (see Figure \ref{fig:N23}). It turns out that this poset is isomorphic to 
the monoid poset $P(\N, \langle 2,3 \rangle)$ where $\N$ is the monoid of natural numbers with the addition and $ \langle 2,3 \rangle = \{2 \alpha + 3 \beta \, \vert \, \alpha, \beta \in \N \}$ is the submonoid of $\N$ spanned by $2$ and $3$.

\begin{theorem}\label{th:monoidnotfull}
The poset $(\mathbb{N},\leq)$ with $a < b$ if and only if $b - a \geq 2$ is a monoid poset but is not a full semigroup poset.
\end{theorem}
\begin{figure}[h]
\centering
\includegraphics[height=.4\textwidth]{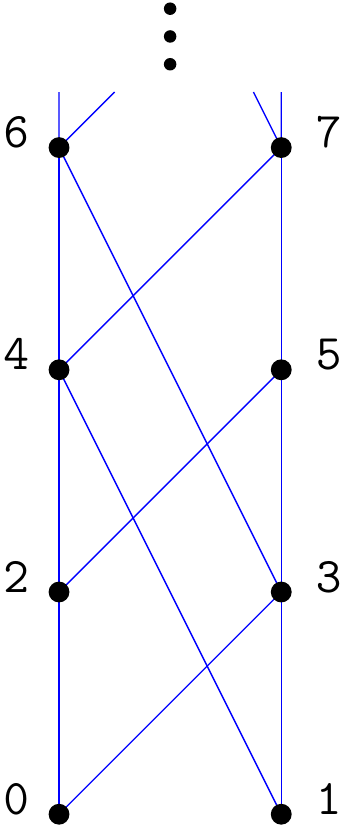}
\caption{$(\mathbb{N},\leq)$ with $a < b \ \Longleftrightarrow \ b - a \geq 2$}
\label{fig:N23}
\end{figure}

\begin{proof}
As we mentioned before, $(\N, \leq)$ is isomorphic to $P(\N, \langle 2,3\rangle)$ and, thus, it is a monoid poset. 
Assume by contradiction that $(\N, \leq)$ is a full semigroup poset. Thus, $\N$ can be endowed with an operation $\cdot$ so that $S = (\N, \cdot)$ is a semigroup and 
$(\N, \leq) \cong P(S,S)$. We first observe that $(\N, \leq)$ has two minima, namely $0$ and $1$. Then, we have that either $0\cdot0 = 0 $ or $0\cdot1 = 0$.

Assume first that $0 \cdot 0 = 0$. Then by Lemma \ref{lemafacil} we have that $0 \cdot y = y$ for all $y \in\ \upa 0 = \N - \{1\}$. Since the poset is full, then $0 \cdot 1 \geq 0$. Moreover,  $0 \cdot 1 \leq 0 \cdot 3  = 3$ and $0 \cdot 1 \leq 0 \cdot 4 = 4$, so we conclude that  $0\cdot 1=0$. 
We also have that $1 \cdot 1 = 1$ because $1 \cdot 1 \geq 1$ and $0 \cdot (1\cdot 1) =  (0\cdot 1)\cdot 1 = 0 \cdot 1 = 0$. So again by Lemma \ref{lemafacil} we have that $1 \cdot y=y$ for all $y \in\ \upa 1 = \N - \{0,2\}$.  Since the poset is full, then $1\cdot 2 \geq 1$. Moreover,  $1\cdot 2\leq 1 \cdot 4 = 4$; so we conclude that $1\cdot 2 \in \{1,4\}$.
  Finally, we have that  $2 = 0 \cdot 2 =  (0\cdot 1)\cdot 2 = 0\cdot (1 \cdot 2) \in \{0 \cdot 1, 0 \cdot 4\} = \{0,4\}$, a contradiction.

Assume now that $0\cdot 0 > 0$, then $0\cdot 1 = 0$. As  a consequence $0\cdot(1\cdot 0) = (0\cdot 1)\cdot 0 = 0\cdot 0 > 0$ and then $1\cdot 0 > 1$. Therefore, we have that  $1\cdot 1 = 1$ and by Lemma \ref{lemafacil} we have that $1\cdot y=y$ for all $y \in \ \upa 1$. Now we have that $1\cdot 2 \geq 1 \cdot 0 > 1$, $1 \cdot 2 \leq 1 \cdot 4 = 4$, so $1 \cdot 2 = 4$. However, this is not possible because $1 \cdot 0 > 1$, $1 \cdot 0 \leq 1 \cdot 2 = 4$ and $1 \cdot 0 \leq 1 \cdot 3 = 3$.
\end{proof}

For every $c \in \N$, one can consider the subposet $(\N_c,\leq)$ of $(\mathbb{N},\leq)$ induced by the interval of integers $[0,c] \cap \N$. If one observes Theorem \ref{th:monoidnotfull}, when proving that $(\N, \leq)$ is not a full semigroup poset the argument only involves the vertices $\{0,1,2,3,4\}$ of the poset. Hence, one can mimic the proof of Theorem \ref{th:monoidnotfull} to get the following result.

\begin{corollary}\label{co:Ncnotfull} $(\N_c,\leq)$ is not a full semigroup poset for all $c \geq 4$.
\end{corollary}

Moreover, for $c = 4$ we have that $(\N_4,\leq)$ is a monoid poset.

\begin{corollary}\label{co:N4semignotfull} $(\N_4,\leq)$ is a monoid poset which is not a full semigroup poset.
\end{corollary}
\begin{proof}
We know that $(\N_4,\leq)$ is not a full semigroup poset by Corollary \ref{co:Ncnotfull}. Moreover $(\N_4,\leq)$ is a semigroup poset $P(N,M)$, where $N = \N_4$, $M = \N_4 - \{1\}$ and with multiplication table as shown in Figure \ref{fig:N4}.
\begin{figure}[h]
\centering
\includegraphics[width=.4\textwidth]{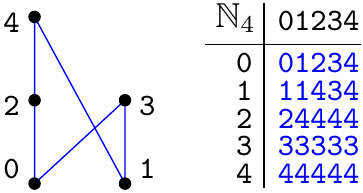}
\caption{($\N_4,\leq)$ is a monoid poset via $P(\N_4,\{0,2,3,4\})$ and not full by Corollary \ref{co:Ncnotfull}.}
\label{fig:N4}
\end{figure}
\end{proof}

For $c \in \N$, one can also consider the subposet of $(\mathbb{N},\leq)$ induced by the set $\{0,2,\ldots,c\}$, which we will denote by $\N_c^*$. Our next goal is to prove that for all $c \geq 6$, then $\N_c^*$ is not a semigroup poset. The poset $\N_6^*$ is depicted in Figure~\ref{fig:diagram}.

\begin{theorem} \label{th:notCayley}
Let $c \geq 6$ and denote by $(\N_c^*,\leq)$ the poset with ground set $\{0,2,\ldots,c\}$ and ordered by $a < b \Longleftrightarrow b - a \geq 2$. Then, $(\N_c^*,\leq)$ is not a semigroup poset.
\end{theorem}
\begin{proof}Assume by contradiction that $(\N_c^*, \leq)$ is a semigroup poset. We observe that $\N_c^*$ has $0$ as global minimum, then by Theorem \ref{thm:characterizationfullmono} we have that $(\N_c^*, \leq)$ is a full monoid poset with $0$ as identity element. We divide the proof depending of the value of $2 \cdot 2$.

{\it Case I: } If $2 \cdot 2 = 2$. One can follow the same ideas as in the proof of Theorem \ref{th:monoidnotfull} to get that this is not possible. More precisely, one can prove that: (i) $2 \cdot x = x$ for all $x \in \ \upa 2$, (ii) $2 \cdot 3 = 2$, (iii) $3 \cdot x = x$ for all $x \in\ \upa 3$, (iv) $3 \cdot 4 \in \{3,6\}$; and conclude that this is not possible since $4 = 2 \cdot 4 = (2 \cdot 3) \cdot 4 = 2 \cdot (3 \cdot 4) \in \{2 \cdot 3, 2 \cdot 6\} = \{3,6\}$.

{\it Case II: } If $2 \cdot 2 \notin \{2,4,5\}$. Since $4,5 \in\ \upa 2$, there exist $x, y \in \N_c^*$ such that $2 \cdot x = 4$ and $2 \cdot y = 5$. However, for all $z \in \N_c^* - \{3\}$ we have that $2 \cdot z \notin \{4,5\}$, a contradiction.

{\it Case III: } If $2 \cdot 2 = 5$. Since $4,6 \in\ \upa 2$, there exist $x, y \in \N_c^*$ such that $2 \cdot x = 4$ and $2 \cdot y = 6$. However, for all $z \in \N_c^* - \{3\}$ we have that $2 \cdot z \notin \{4,6\}$, a contradiction.

{\it Case IV: } $2 \cdot 2 = 4$.  We are going to prove by induction that $2 \cdot x = 2 + x$ for all $x \in \{0,2,\ldots,c-2\}$. The result holds for $x = 0$ and $x = 2$. Assume now that the result holds for $x \in \{0,2,\ldots,k-1\}$ and let us prove it for $x = k \in \{3,\ldots,c-2\}$. Since $k+2 \in \ \upa 2$, then there exists $x \in \N_c^*$ such that $2 \cdot x = k+2$. By induction hypothesis we have that $2 \cdot y = y+2 \not= k+2$ for $y \in \{0,2,\ldots,k-1\}$. Moreover, $2 \cdot (k-1) = k+1$ and, hence, for $y \in\ \upa (k-1) = \{k-1,k+1,\ldots,c\}$ we have that $2 \cdot y \geq k+1$. Hence we can only have that $2 \cdot k = k +2$. Finally we have that $2 \cdot (c-2) = c$ and $2 \cdot (c-3) = c-1$,  but this implies that $2 \cdot c \geq 2 \cdot (c-2) = c$ and $2 \cdot c \geq 2 \cdot (c-3) = c-1$, which is not possible because $c-1$ and $c$ are the two maxima of $(\N_c^*,\leq)$.
\end{proof}

Note that similar to the definitions of $(\N_c,\leq)$ and $(\N^*_c,\leq)$ there is a natural set of types of subposets of $P(\Z, \langle 2,3 \rangle)$ obtained from selecting all points between up to two chosen maxima and up to two chosen minima. We believe that it is interesting to study these posets. Probably, large enough posets of a given type all behave the same with respect to their Cayley properties.

We finish  this section considering $(N_i)_{i \in \N},$ a family of full semigroup posets that are not monoid posets. The poset $N_i$ is obtained from the N-poset by adding $i$ new elements and the cover relations $i < i-1 < \cdots < 1 < a$ (see Figure \ref{fig:Ni}).

\begin{theorem} \label{th:fullnotmonoid}
The poset $N_i$ described in Figure \ref{fig:Ni} is a full semigroup poset for all $i \geq 1$ and is not a monoid poset for all $i \geq 2$.
\begin{figure}[h]
\centering
\includegraphics[height=.3\textwidth]{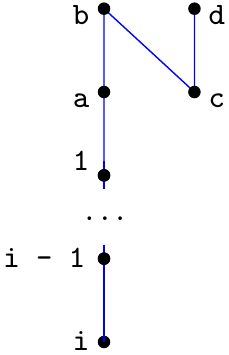}
\caption{Drawing of the poset $N_i$.}
\label{fig:Ni}
\end{figure}
\end{theorem}
\begin{proof}
Figure \ref{fig:tablaN1} shows a multiplication table witnessing that $N_1$ is a full semigroup poset. 

\begin{figure}[h]
\centering
\includegraphics[height=.2\textwidth]{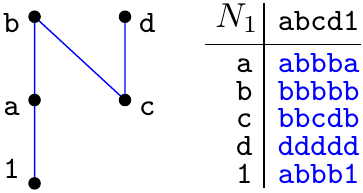}
\caption{$N_1$ is a full semigroup poset.}
\label{fig:tablaN1}
\end{figure}

By recurrence we define the multiplications involving the element $i$ (see Figure \ref{fig:tablaNi}): 
\begin{itemize} \item $x \cdot i = x \cdot i$ and $i \cdot x = 1 \cdot x$ for $x \in \{a,b,c,d\}$, and
\item $i \cdot n = n \cdot i = n$ for all $n \in \{1,\dots,i\}$.
\end{itemize}

\begin{figure}[h]
\centering
\includegraphics[height=.3\textwidth]{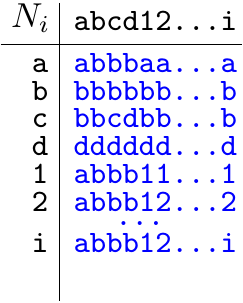}
\caption{Cayley table witnessing that $N_i$ is a full semigroup poset. }
\label{fig:tablaNi}
\end{figure}

It is straightforward to check that $N_i$ is a full semigroup poset with this multiplication table.
Moreover let $i \geq 2$ and assume that $N_i \cong P(N,M)$ is a monoid poset. By Theorem \ref{thm:characterizationmono}, $M$ corresponds to 
a principal upset $\upa e$ where $e$ is the identity element of $N$. Moreover, for the element $i \in N_i$ there is a map $\varphi_i$ such that $\varphi_i(\upa e) = \upa i$. Then, in particular, the cardinality of $\upa e$ is greater or equal to the one of $\upa i$, which is $i+2$. Hence, we deduce that $e = i$. Now,  we observe that $c \leq b$ and $c\leq d$. Then,  there exist $x,y \in\ \upa i$, such that $b =c \cdot x$ and $d = c \cdot y$.
Nevertheless, $\upa i$ is a chain, so either $x \geq y$ or $y \geq x$. If $x \geq y$ we get that $b = c \cdot x \geq c \cdot y = d$, a contradiction.  If $x \leq y$ we get that $b = c \cdot x \leq c \cdot y = d$, a contradiction too. Thus $N_i$ is not a monoid poset.
\end{proof}

\section{Constructions}\label{sec:construction}
In the present section we investigate the behavior of being a semigroup poset under standard poset and semigroup operations.
A very simple result of this kind is the following

\begin{observation}\label{obs:min}
 If $S$ is a semigroups such that $P \cong P(S,S)$ is full, then the poset $\check{P}$ obtained by adding a minimum is a full monoid poset $\check{P}=P(S\cup\{e\},S\cup\{e\})$, where $S\cup\{e\}$ is the monoid obtained from $S$ by adjoining the neutral element $e$.
\end{observation}

Another similar operation is the following:
\begin{observation}\label{obs:max}
 If $S<T$ are semigroups such that $P \cong P(T,S)$ is a (full) semigroup poset, then the poset $\hat{P}$ obtained by adding a maximum is a (full) semigroup poset $\hat{P}=P(T\cup\{a\},S\cup\{a\})$, where $S\cup\{a\}$ is the semigroup obtained from $S$ by adjoining an absorbing element $a$, i.e., $at=ta=a$ for all $t\in T$.
\end{observation}

The \emph{Cartesian product} of two posets $P\times P'$ is defined on the product set by setting $(x,x')\leq (y,y')$ if and only if $x\leq y$ and $x'\leq y'$. The Cartesian product of two semigroups is just defined by componentwise operation.

\begin{observation}
 If $P$ and $P'$ are semigroup posets, then $P\times P'$ is a semigroup poset. Moreover, $P \times P'$ is full (resp. a monoid poset) whenever both $P$ and $P'$ so are. 
\end{observation}
\begin{proof}
 If $P\cong P(N,M)$ and $P'\cong P(N',M')$, then $P\times P'\cong P(N\times N',M\times M')$ and $M\times M'<N\times N'$. We have that $M \times M' = N \times N'$ if and only if $M = N$ and $M' = N'$. Moreover, $N \times N'$ is a monoid and $M \times M'$ is a submonoid of it if both $N$, $N'$ are monoids and $M < N, M' < N'$ are submonoids. 
\end{proof}

For semigroups $S<T$ we call a semigroup-endomorphism $\sigma:T \to S$ of $T$ a \emph{retract} if $\sigma(T)=S$ and for the restriction to $S$ we have $\sigma_{|S}=id_{S}$. Note that if $P \cong P(T,S)$ and $\sigma:T \to S$ a retract, then $\sigma$ also is an order endomorphism of $P$.

\begin{proposition}\label{prop:retract}
 Let $S<T$ be semigroups such that $P \cong P(T,S)$ and let $\sigma:T \to S$ be a retract, then there is a semigroup $T'$ such that $P \cong P(T',T')$, i.e., $P$ is full.
\end{proposition}
\begin{proof}
 Let $P \cong P(T,S)$ and let $\sigma:T \to S$ be a retract. Define a new operation on $T$ by $t\cdot t':=t\sigma(t')$ and call the new semigroup $T'$.
 
 First, we check that $\cdot$ is associative. Transform $t\cdot(t'\cdot t'')=t\sigma(t'\sigma(t''))$. Since $\sigma$ is a homomorphism and since it is a retract it is idempotent, the latter equals $t\sigma(t')\sigma(t'')=(t\cdot t')\cdot t''$.
 
 Now, observe that since $\sigma$ is the identity on $S$ we have $t=t's\iff t=t'\cdot s$, i.e., both orders are the same. 
 
 We conclude that $P \cong P(T',T')$ is full.
\end{proof}

An element $x$ of a semigroup $T$ is called \emph{irreducible} if $x=ab\implies x\in\{a,b\}$ for all $a,b\in T$. Note that for $P \cong P(T,S)$ the set of irreducibles of $T$ is a subset of $S\cup \Min(P)$, since if $y<x$ and $x\notin S$, then there is $s\in S$ such that $ys=x$. Furthermore, $x$ is \emph{self-centered} if $yx=x\iff xy=x$ for all $y\in T$. Note that if $x$ commutes with every element of $T$ or if it is (right and left) cancellative, then $x$ is self-centered.


Let $x\in P$ a poset element and $Q$ another poset, we denote by $P_xQ$ the \emph{blowup} of $x$ by $Q$, which is the poset where $x$ has been replaced by a copy of $Q$ and all elements of $Q$ behave with respect to the elements of $P\setminus x$ as $x$ did (see Figure \ref{fig:exblowup} for an example).

\begin{figure}[h]
\centering
\includegraphics[height=.25\textwidth]{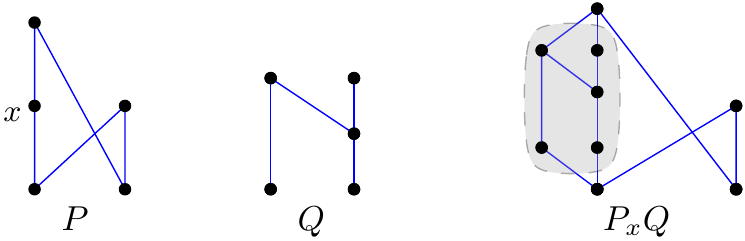}
\caption{Example of a blowup.}
\label{fig:exblowup}
\end{figure}

\begin{theorem}\label{thm:blowup}
 Let $P$ and $Q$ be semigroup posets and $x\in P$ irreducible and self-centered. If $x\in \Min(P)$ or $Q$ is full, then $P_xQ$ is a semigroup poset. If $P$ and $Q$ are full, then $P_xQ$ is full. If $P$ is a monoid poset with neutral element $e_P\neq x$ or $Q$ is a monoid poset as well, then $P_xQ$ is a monoid poset.
\end{theorem}
\begin{proof}
 Let $P \cong P(T,S)$ and $Q \cong P(V,U)$. We will show that $P_xQ \cong P(T\cup V\setminus x, S\cup U\setminus x)$ where the new operation is defined as $$t\cdot t'=\begin{cases}
                                                                           tt'  & t,t'\in T \text{ or } t,t'\in V,\\
                                                                           xt'  & t\in V, t'\in T \text{ and } xt'\neq x,\\
                                                                           tx  & t'\in V, t\in T \text{ and } tx\neq x,\\
                                                                           t'  & t'\in V, t\in T \text{ and } tx=x,\\
                                                                           t  & t\in V, t'\in T \text{ and } xt'=x.\\
                                                                          \end{cases}$$
Since $x$ is irreducible, the operation is well defined.

We start by arguing, that the above operation is a semigroup. In words the above operation does the following: If the elements come both from the same set, i.e., either $T$ or $V$, then their composition is unchanged. If one comes from $T$ and one from $V$, then replace the one from $V$ with $x$, unless this results in $x$. In the latter case just replace the product by the element of $V$. So, let $t,t',t''\in V\cup T$. We want to show $(t\cdot t')\cdot t''=t\cdot (t'\cdot t'')$. If all three are in $V$ or all three are in $T$, we clearly have associativity. The remaining cases have to be considered individually.

\begin{case}\label{case:TTV}
 $t,t'\in T$ and $t''\in V$.
\end{case}
\begin{subcase}
 $tt'x=x$.
\end{subcase}
This implies $(t\cdot t')\cdot t''=t''$. On the other hand, irreducibility of $x$ and $t(t'x)=x$ implies $(t'x)=x$ and $t\cdot (t'\cdot t'')=t\cdot t''$. But since $tx=t(t'x)=x$, also $t\cdot t''=t''$.
\begin{subcase}
 $tt'x\neq x$.
\end{subcase}
This implies $(t\cdot t')\cdot t''=tt'x$. On the other hand, if $t'x\neq x$, then $t\cdot (t'\cdot t'')=t\cdot (t'x)=tt'x$. If $t'x=x$, then $tt'x\neq x$ implies $tx\neq x$ and $t\cdot t''=tx=tt'x$.

\begin{case}\label{case:VTT}
  $t,t''\in T$ and $t'\in V$.
\end{case}
\begin{subcase}\label{subcase:==}
 $tx=x$ and $xt''=x$.
\end{subcase}
We have $(t\cdot t')\cdot t''= t'\cdot t''=t'=t\cdot t'=t\cdot (t'\cdot t'')$.

\begin{subcase}\label{subcase:=neq}
 $tx=x$ and $xt''\neq x$.
\end{subcase}
We have $(t\cdot t')\cdot t''= t'\cdot t''=xt''=t(xt'')=t\cdot (t'\cdot t'')$.

\begin{subcase}
 $tx\neq x$ and $xt''= x$.
\end{subcase}
We have $(t\cdot t')\cdot t''= (tx)t'=tx==t\cdot t'=t\cdot (t'\cdot t'')$.

\begin{subcase}
 $tx\neq x$ and $xt''\neq x$.
\end{subcase}
We have $(t\cdot t')\cdot t''=txt''=t\cdot (t'\cdot t'')$.

\begin{case}
 $t\in V$ and $t',t''\in T$.
\end{case}
This case works analogous to Case~\ref{case:TTV}.

\begin{case}\label{case:TVV}
 $t\in T$ and $t',t''\in V$.
\end{case}
\begin{subcase}
 $tx=x$
\end{subcase}
We have $(t\cdot t')\cdot t''=t'\cdot t''=t\cdot (t'\cdot t'')$.

\begin{subcase}\label{subcase:txx}
 $tx\neq x$
\end{subcase}
Note that $txx=x$ implies $x^2=x$ by irreducibility of $x$ and hence $tx=x$, i.e., $tx\neq x\implies txx\neq x$ and we compute $(t\cdot t')\cdot t''=(tx)\cdot t''=tx=t\cdot (t'\cdot t'')$.

\begin{case}\label{case:VVT}
 $t,t''\in V$ and $t'\in T$.
\end{case}
\begin{subcase}
 $xt'=x$
\end{subcase}
Since $x$ is self-centered, we have $t'x=x$. Thus, we compute $(t\cdot t')\cdot t''=t\cdot t''=t\cdot (t'\cdot t'')$.

\begin{subcase}
 $xt'\neq x$
\end{subcase}
Again, since $x$ is self-centered, we have $t'x\neq x$. Note furthermore that by $x$ being self-centered $xt'x=x$ implies $t'xx=x=xxt'$ and as argued in Case~\ref{subcase:txx} this yields $t'x=x=xt'$. Thus, in the present case we have $xt'x\neq x$ and we compute $(t\cdot t')\cdot t''=(xt')\cdot t''=xt'x=t\cdot(t'x)=t\cdot (t'\cdot t'')$.

\begin{case}
 $t,t'\in V$ and $t''\in T$.
\end{case}
This case works analogous to Case~\ref{case:TVV}.

\bigskip

Let us now see, that all order relations of $P_xQ$ are realized by the new operation. Let $y<z$ be two elements of $P_xQ$. If both $y,z\in T$ or both $y,z\in V$, then clearly this is the case. If $y\in T$ and $z\in V$, then $y<x$ in $P$, so $ys=x$ for some $s\in S$. Since $y\neq x$ and $x$ is irreducible, this implies that $x=s$, i.e., $yx=x$. Thus, $y\cdot z=z$ after the fourth part of ``$\cdot$''. Since $x\notin \Min(P)$, we have that $Q$ is full, i.e., $z\in U$ and $y<z$.
If $y\in U$ and $z\in T$, then there is an $s\in S$ so that $xs=z$ in $P$ and we have $ys=z$ by the second part of ``$\cdot$''.

Now we show that no new comparabilities can arise. Let $t\in T\cup V\setminus\{x\}$ and $s\in S\cup U\setminus\{x\}$ and $t'=t\cdot s$. We have to show that $t\leq t'$ in $P_xQ$. This clearly holds if $t\in T$ and $s\in S$ or $t\in V$ and $s\in U$. 

If $t\in T$ and $s\in U$, then $t'=tx$ if $tx\neq x$. Since $x\in S$, this yields $t<t'$ in $P_xQ$. If $tx=x$, then $t'=s$, but again since $x\in S$ by $tx=x$ we have $t<x$ in $P$ and thus $t<s$ in $P_xQ$.

If $t\in V$ and $s\in S$, then $t'=xs$ if $xs\neq x$ by the second part of ``$\cdot$'', i.e., $x<t'$ in $P$ and thus $t<t'$ in $P_xQ$.
If $xs=x$, then $t'=t$ by the fifth part of ``$\cdot$'' and we clearly have $t\leq t'$ in $P_xQ$.

It remains to show the last part of the proposition.
Using the the characterization theorems from Section~\ref{sec:characterization} we get the following: Since $x\in \Min(P)$ or $Q$ is full (and $x\in S$), we have that $S\cup U$ is an upset of $P_xQ$. Thus, if both are full, then so is $P_xQ$. Finally, if $P$ is a monoid and $x\neq e_P$, where $e_P$ is the neutral element of $P$, then in the new operation $e_P$ is the neutral element for $P_xQ$.
However, if $x=e_P$, but $Q$ is a monoid with neutral element $e_Q$, then $e_Q$ is the neutral element for $P_xQ$.
\end{proof}

Note that if $x,y$ are irreducible and self-centered in $P$, then $y$ is irreducible and self-centered in $P_xQ$ and; moreover, the operations defined following in
Theorem \ref{thm:blowup} in $(P_xQ)_yR$ and $(P_yR)_xQ$ coincide. In particular, one can easily come up with an operation defined for the blowup of a (possibly infinite) set of elements of $P$, each replaced by a different $Q$. We will skip the proof of this result, since it is similar to the one of Theorem~\ref{thm:blowup}. An interesting consequence of Theorem~\ref{thm:blowup} is when considering the \emph{modular decomposition} of a poset $P$, i.e., a \emph{module} of $P$ is an induced subposet $Q$ such that all elements of $Q$ behave in the same way with respect to elements of $P\setminus Q$. This yields a recursive decomposition of $P$. Theorem~\ref{thm:blowup} can be used to formulate minimality of non-Cayley posets with respect to modular decompositions. In particular, it lead us to consider the posets in the beginning of Section~\ref{sec:firstseparations} since they are \emph{prime} with respect to modular decomposition and series-parallel posets in Corollary~\ref{cor:seriesparallel}, because their modular decompositions are as simple as possible.

Given a join semilattice, one obtains an \emph{antichain-blowup of a semilattice} by replacing some of its join-irreducible elements by antichains. See Figure~\ref{fig:blowup} for an example of this construction.

\begin{figure}[h]
\centering
\includegraphics[height=.4\textwidth]{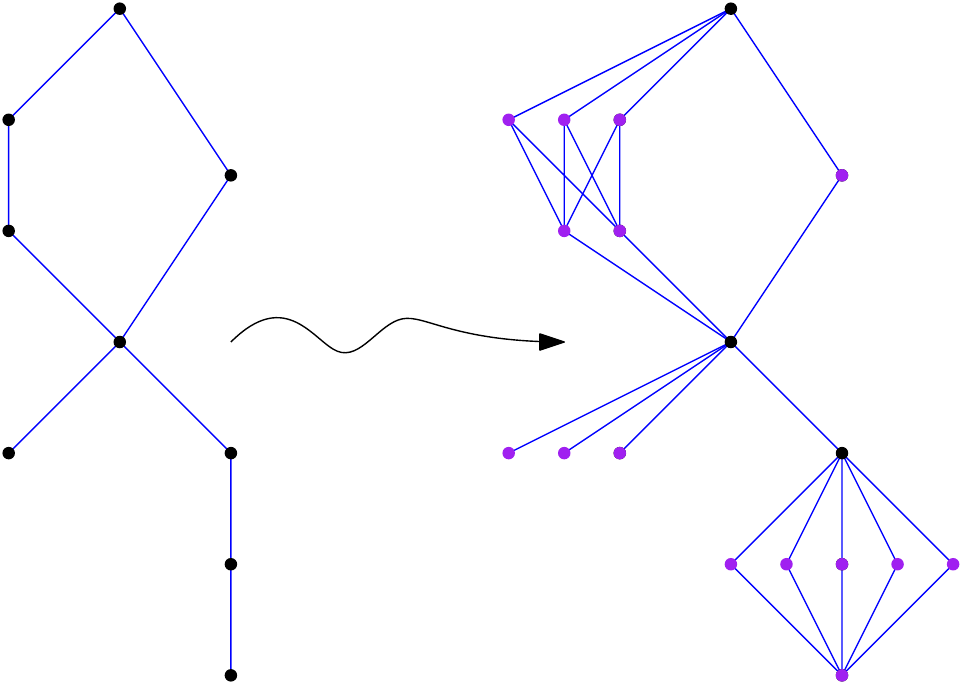}
\caption{An antichain-blowup of a join-semilattice}
\label{fig:blowup}
\end{figure}

Since antichains can be easily realized as full posets and join-semilattices are full (with the join operation), we obtain the following result as
a special case of Theorem~\ref{thm:blowup}. 
\begin{corollary}
 Antichain-blowups of join-semilattices are full.
\end{corollary}

Given two posets $P,Q$ their \emph{series composition} $P*Q$ is the poset obtained by taking the disjoint union of $P$ and $Q$ and making all elements of $P$ inferior to all elements of $Q$. The \emph{parallel composition} of $P+Q$ of $P$ and $Q$ is obtained by taking the disjoint union of $P$ and $Q$ and not adding any further comparabilities.

\begin{proposition}\label{prop:seriesparallel}
 Let $P, Q$ be semigroup posets, then $P+Q$ is a semigroup poset. If $P,Q$ are both full, then so is $P+Q$. 
 If $P \cong P(N,M)$ and $Q \cong P(N',M')$ are monoid posets and there is a monoid homomorphism $\sigma:N\to N'$ with $\sigma(M)=M'$, then $P+Q$ is a monoid poset.
 
 If $Q$ is full, then $P*Q$ is a semigroup poset. If additionally $P$ is full or a monoid poset, then so is $P*Q$.
  
\end{proposition}
\begin{proof}
 Let $P \cong P(T,S)$ and $Q \cong P(T',S')$. Now consider $P(T\cup T',S\cup S')$ where for $t\in T$ and $t'\in T'$ we define $tt'=t$ and $t't=t'$. It is easy to see, that this gives a semigroup generating $P+Q$. Clearly, if both $P,Q$ are full, then so is $P+Q$.
 
 Let $P \cong P(N,M)$ and $Q \cong P(N',M')$ monoids and $\sigma:N\to N'$ a monoid homomorphism $\sigma:N\to N'$ with $\sigma(M)=M'$. We show that $P+Q \cong P(N\cup N', M)$ with respect to the operation defined as $$t\cdot t'=\begin{cases}
                                                                           tt'  & t,t'\in N \text{ or } t,t'\in N',\\
                                                                           t\sigma(t')  & t\in N', t'\in N \text{ and } xt'\neq x,\\
                                                                           \sigma(t)t'                                     & t\in N, t'\in N'\\
                                                                           
                                                                          \end{cases}$$
 Since $\sigma$ is a monoid homomorphism it follows easily that ``$\cdot$'' is a monoid. Moreover, the order relation is easily seen to be the right one since $\sigma(M)=M'$ and we only take $M$ as submonoid generating the order.

 Let $P \cong P(T,S)$ and $Q \cong P(S',S')$. Now consider $P(T\cup S',S\cup S')$ where for $t\in T$ and $t'\in S'$ we define $tt'=t't=t'$. It is easy to see, that this gives a semigroup generating $P*Q$. Clearly, if $P$ is full or monoid, then so is $P*Q$.
\end{proof}

A poset is called \emph{series-parallel} if it can be constructed from a singleton by series and parallel compositions. 
Hence, a particular consequence of Proposition~\ref{prop:seriesparallel} is:
\begin{corollary}\label{cor:seriesparallel}
 Series parallel posets are full.
\end{corollary}

For a poset $P$ and a positive integer $k$, we denote by $kP$, the poset in which every element of $P$ is replaced by an antichain of size $k$, as in the blow-up operation above.

\begin{proposition}\label{prop:blowupmono}
 If $P$ is a monoid poset $P(N,M)$ for monoids $M<N$, where $nm=n\Rightarrow m=e$ for all $n\in N, m\in M$, then for any $k\geq 1$ the poset $kP$ is a monoid poset.
\end{proposition}
\begin{proof}
 Set $N'=N\times\mathbb{Z}_k$ and $M'=((M\setminus\{e\})\times\mathbb{Z}_k)\cup\{(e,0)\}$. Then $P\cong P(N',M')$.
\end{proof}

Another operation that we have decided not to describe in detail here is the one of taking semigroup quotients and how this affects the poset structure. Some operations of this type will be used in Section~\ref{sec:autoequivalent}, though.

%

\section{Weak orders}\label{sec:examples}

A poset $P$ is called a \emph{weak order} if it is a (possibly infinite) chain of antichains called \emph{levels} $(\ldots, A_i,A_{i+1},\ldots)$ such that $x< y\iff x\in A_i$ and $y\in A_j$ with $i<j$.

\begin{figure}[h]
\centering
\includegraphics[height=.4\textwidth]{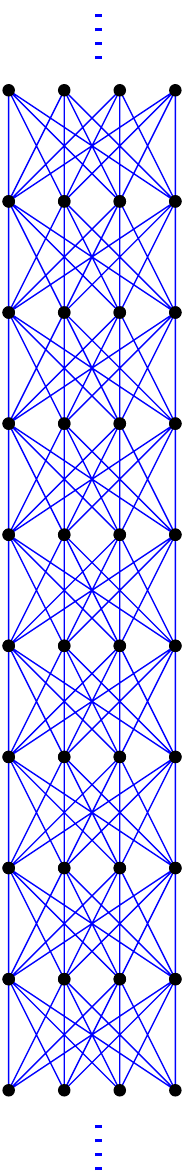}
\caption{The bi-infinite weak order with levels of size $4$.}
\label{fig:weak}
\end{figure}

Since any chain is a monoid poset but also a full poset and antichains are full posets, as applications of Observation~\ref{obs:min}, Theorem~\ref{thm:blowup}, and Proposition~\ref{prop:blowupmono} is the following for which we however give an independent proof:
\begin{proposition}\label{prop:weak}
Any weak order $P$ is a full semigroup poset. If $P$ has a minimum or all levels are of the same size, then $P$ is a monoid poset
\end{proposition}
\begin{proof}
 To see the first part of the claim just identify every element of $P$ with the tuple $(i,j)$, where $i$ is the index of its level and $j$ its positions within the level. We define $(i,j)\cdot(k,\ell)$ by $(\max(i,k),m)$, { where $m=j$ if $i > k$, $m=\ell$ if $i < k$, and $m= {\rm max}(j,\ell)$ otherwise}. It is straightforward to check that $\cdot$ is associative. Thus, $P$ is full.
 Clearly, by Theorem~\ref{thm:characterizationfullmono} if $P$ has a minimum, then it is a monoid poset. 
 
 If all levels of $P$ are of the same size $k$, then let $M$ be any monoid of size $k$ with neutral element $e$. We can represent $P$ as $P(\mathbb{Z}\times M, \mathbb{N}\setminus\{0\}\times M\cup \{(0,e)\})$, if $P$ is bi-infinite. Otherwise, we can represent a chain with minimum by $\mathbb{N}$ instead of $\mathbb{Z}$. If there is a maximum layer, then we replace the operation in the first component by addition followed by taking a maximum.
\end{proof}

In the rest of the section we want to study if further weak orders apart from those mentioned in Proposition~\ref{prop:weak} can be monoid posets. For this we establish a couple of more general necessary conditions on monoid posets. 

\begin{lemma}\label{lem:numbercomparable}
 Let $P$ be a monoid poset and let $x \in P$ such that $x \leq e$ . There is an injective order homomorphism from 
$P \setminus (\upa x\cup \downa x)$ to $P \setminus (\upa e\cup \downa e)$
\end{lemma}
\begin{proof}We will use that by Theorem~\ref{thm:characterizationmono} left-multiplication by $x$ coincides with the order endomorphism $\varphi_x$. Since $x \leq e$ and
$\varphi_x: \upa e \rightarrow \upa x$ is onto, there exists $x' \in \upa e$ such that $xx' = e$. Let  $y,z\in P \setminus (\upa x\cup \downa x)$ be different elements. Since $\varphi_x \in \End(P)$, we have that 
\begin{itemize}
\item $x'y$ is not comparable with $e$. Indeed, if $x' y \leq e$ (or $x'y \geq e$), then $y = x x' y = \varphi_x(x' y) \leq \varphi_x(e) = x$ (or
$y = x x' y = \varphi_x(x' y) \leq \varphi_x(e) = x$),  which contradicts $y\parallel x$; and,
\item  $x'y \neq x'z$ for all $i \neq j$; otherwise $y = \varphi_x(x'y) = \varphi_x(x'z) = z$. 
\end{itemize}
This proves the result.
\end{proof}

%
%
%

\begin{lemma}\label{lem:upwards}
 Let $P$ be a monoid poset, $x\in P$, and $A\subseteq \upa x$ an antichain. If $\varphi_x^{-1}(A):=B$ is a maximal antichain of $\upa e$, then $\varphi_x:\downa B ~ \cap\upa e \to \downa A~ \cap \upa x$ is onto. 
\end{lemma}
\begin{proof}

Let us show that if $y\in \downa B \cap \upa e$, then $\varphi_x(y) \in \downa A \cap \upa x$.
Since $y\in \downa B \cap \upa e$, there is $b \in B$ with $y\leqslant b$. Since $\varphi_x\in \End(P)$ this implies $\varphi_x(y) \leqslant \varphi_x(b)\in A$. Thus, $\varphi_x(y) \in \downa A$.
Moreover since $x\leqslant e$ using Theorem~\ref{thm:characterizationmono} we get $\varphi_x : \upa e \rightarrow \upa x$. More precisely, since $y \in \upa e$ we have $ \varphi_x(y)\in \upa x$ and $\varphi_x(y) \in \downa A \cap \upa x$.

It remains to show that $\varphi_x:\downa B\cap\upa e \to \downa A\cap\upa x$ is surjective.
Let $z\in \downa A \cap \upa x$. Since $ \varphi_x :\upa e \hookrightarrow \upa x$ is onto, there exists $w \in \upa e$ with $\varphi_x(w) = z$. If there is $b\in B$ with $w\leq b$ we are done. Otherwise, since $B$ is a maximal antichain there is $b\in B$ with $b \leq w$. Thus, $\varphi_x(b) \leq \varphi_x(w)$. Now, $ \varphi_x(b) \in A$ and $\varphi_x(w)=z \in \downa A$ imply $\varphi_x(b) = \varphi_x(w) = z$ and in particular $z\in A$. Thus, $w\in B$ and we are done
\end{proof}

Lemmas~\ref{lem:numbercomparable} and~\ref{lem:upwards} give.
\begin{proposition}
  If a weak order without minimum is a monoid poset such that $e\in A_0$, then $\mathrm{sup}\{\ell_i\mid i\geq k\}$ is independent of $k$ and $\ell_i\leq\ell_0$ for all $i\leq 0$.
\end{proposition}

The above proposition in particular yields that infinite weak orders with a maximum cannot be monoid posets, see the example in Figure~\ref{fig:diagram}. More interestingly, weak orders without minimum in which a level of maximum size appears only finitely many times, are not monoid posets. In particular, if one level is of size $2$ and all others are of size $1$, the corresponding poset depicted in Figure~\ref{fig:diagram} is not a monoid poset, but a (full) semigroup poset. We believe the following to be true.

\begin{conjecture}\label{conj:weak}
A bi-infinite weak order is a monoid poset if and only if all its levels are of the same size.
\end{conjecture}

\section{Auto-equivalent posets}\label{sec:autoequivalent}
In the present section we extend results of~\cite{Cha-15} by characterizing classes of so-called auto-equivalent posets as Cayley posets of certain group-embeddable abelian monoids. More precisely, we will characterize Cayley posets of { finitely generated} pointed submonoids  abelian groups. For the particular case when $ M$ is a submonoid of $\Z^m$, these posets were characterized in~\cite[Theorem 5.5]{Cha-15}. 


Let us start with the introduction of the monoids that we consider. We will mostly consider submonoids of groups. 
If $M$ is a submonoid of a group $G$, 
then the binary relation $\lS$ of $P(M,M)$ can be rewritten as: 
\begin{center} $s \lS t\, \Longleftrightarrow\, { s^{-1}t} \in  M$. \end{center} Since $M$ is a monoid, a direct application of Proposition \ref{prop:sunital} yields that $\lS$ is an order relation if and only if the acyclicity condition is satisfied. In this setting this is equivalent to the neutral element $e$ of $M$ being its only invertible element, i.e., $ M \cap M^{-1} = \{e\}$. This property of $M$ is called {\it pointed}. 
The second property that we require for $M$ is that $M$ is finitely generated and the group $G$ containing it is abelian. Monoids of this kind have been considered extensively due to their connections with (binomial) lattice ideals (see, for example, \cite{BCMP,Eis-96,Mil-05}). A feature of these monoids, that we will make use of, is that they have a unique minimal set of generators. It coincides with the set of \emph{atoms} of $P(M,M)$, i.e., the minimal elements of $ M \setminus \{e\}$ with respect to $\lS$.

%
%
%
%

Let us now give the necessary definitions to describe the resulting class of posets. 
We say that a poset $P$ is \emph{auto-equivalent} if $P$ has a global minimum and there exists a commutative submonoid $T < \End(P)$ such that for
every $x \in P$ there exists a unique $\varphi_x \in T$  such that $\varphi_x(P) =\, \upa x$ and $\varphi_x: P \rightarrow\, \upa x$ is an order isomorphism.
Note that this definition corresponds to the one given in~\cite{Cha-15}.
Finally, a poset $P$ is said to be \emph{locally finite} if for all $x, y \in P$, there is a finite number of elements in the interval $[x,y] := \{z \in P \, \vert \, x \leq z \leq y\}$. 



\begin{lemma}\label{lem:monotopo}
 Let $M$ be an acyclic abelian monoid. We have that $P(M,M)$ is auto-equivalent. Moreover, if $M$ is a finitely generated submonoid of a group, then $P(M,M)$ is locally finite and has a finite number of atoms.
\end{lemma}
\begin{proof}
  The poset $P(M,M)$ has the identity element $e \in  M$ as minimum. Taking $\varphi_s:  M \rightarrow  M$ as $\varphi_s(s') = s + s'$ we have that $\varphi_e$ is the identity map on $M$ and $\varphi_y \circ \varphi_z = \varphi_z \circ \varphi_y = \varphi_{y+z}$, for all $y,z \in  M$. Thus, $P(M, M)$ is auto-equivalent.
  
  If $M$ is finitely generated, then its unique minimal (finite) set of generators $A$ corresponds to the atoms of $P(M,M)$. Moreover, since $M$ is acyclic and the submonoid of a group, $M$ is pointed. Hence, a standard argument yields that $P(M, M)$ is locally finite, see e.g.~\cite{BCMP}.
\end{proof}

\begin{lemma}\label{lem:potomono}
 Let $P$ be a locally finite, auto-equivalent poset. We have that there is a pointed submonoid $M$ of an abelian group such that $P \cong P(M,M)$. Moreover, if $P$ has a finite number of atoms, then $M$ is finitely generated.
\end{lemma}
\begin{proof}
 Denote by $A$ the set of atoms of $P$ and take $T = \{\varphi_x\}_{x \in P}$ a commutative submonoid of ${\rm End}(P)$ such that $\varphi_x(P) =\, \upa x$ and $\varphi_x: P \rightarrow\, \upa x$ is an order isomorphism.

We are going to associate to $P$ a subgroup $L_{P} \subseteq \Z^A$, where $\Z^A$ is the set of mappings from $A$ to $\Z$ with finite support. Defining $\N^A$ similarly, we consider $f: \N^A \longrightarrow  P$ defined inductively as $f(0) = e$ and $f(\alpha + e_a) = \varphi_{a}(f(\alpha))$ for all 
 $\alpha \in \N^A$, where $e_a(x)=1$ if $x=a$ and $e_a(x)=0$ otherwise for all $x\in A$. In particular, $f(e_a) = \varphi_{a}(f(0)) = \varphi_{a}(e) = a$ for all $a \in A$.

It is not difficult to check that $f$ is well defined (because $P$ is auto-equivalent) and surjective (because $P$ is locally finite) and the set $L_{P} := \{\alpha - \beta  \in \Z^A \, \vert \, f(\alpha) = f(\beta) \}$ is a subgroup of $\Z^A$ (see \cite[Section 5]{Cha-15}). 
We set $G=\Z^A / L_{P}$ and $M=\N^A / L_{P}$, where the latter is just the subset of equivalence classes that have non-empty intersection with $\N^A$. Clearly, $M$ is a submonoid of the abelian group $G$, and if $A$ is finite, then $G\cong\Z^m \oplus T$, where $m = {\rm rk}(\Z^{|A|} / L_{P})$ and $T$ is a finite abelian group, i.e., $G$ is finitely generated. Furthermore, $M$ is pointed and $(P, \leq)$ and $P(M,M)$ are isomorphic. More precisely, 
$$
\begin{array}{cccl}
\psi: & P & \longrightarrow &  M \\
& x & \longmapsto & \sum_{a\in A} \alpha(a) e_a, \text{ if  } f(\alpha) = x
\end{array}
$$
is an order isomorphism and $M$ is generated by $\psi(A)$.
\end{proof}

An immediate consequence of Lemmas~\ref{lem:monotopo} and~\ref{lem:potomono} is the following:

\begin{theorem}\label{equivasemigrupo}
A poset $(P, \leq)$ is isomorphic to $P(M,M)$ for a  { finitely generated} pointed submonoid $M$ of an abelian group if and only if $P$ is auto-equivalent, locally finite, and has a finite number of atoms.
\end{theorem}

Note that local finiteness is essential in the previous result, as the following result shows. Here, by $\leq_{\mathrm{lex}}$ we denote the lexicographic order.
\begin{proposition}
 The poset $P=(\N^2, \leq_{\mathrm{lex}})$ is auto-equivalent, has exactly one atom, and is an abelian monoid poset. However, there is no group-embeddable $M$ such that $P \cong P(M,M)$.
\end{proposition}
\begin{proof}
Clearly, $P$ has just one atom.
 Let us show that $P \cong P(M,M)$ for an abelian monoid.  Indeed, define $$(i,j)\cdot (k,\ell)=\begin{cases}
 (i,j) & i>k,\\
 (k,\ell) & i<k,\\
 (i,\max(j,\ell)) & i=k.
 \end{cases}$$
 It is straightforward to check that $M$ is an abelian monoid such that $P \cong P(M,M)$. By Lemma~\ref{lem:monotopo} $P$ is auto-equivalent.
 
 Let $M$ be a monoid on $\N^2$ such that $P \cong P(M,M)$. We show that $M$ cannot be cancellative and therefore is not the submonoid of a group. Let $a\in\{0\}\times\N$ and consider $a\cdot(1,0)$. 
 If $a\cdot(1,0)=(0,m)<(1,0)$, then since left-multiplication is a endomorphism there are $(0,k),(0,\ell)$ such that $a\cdot(0,k)=a\cdot(0,\ell)\leq (0,m)$ and $M$ is not cancellative.
 If $a\cdot(1,0)>(1,0)$, then there is some $(0,m)$ such that $a\cdot(0,m)=(1,0)$, but then the elements between $a$ and $(0,m)$ do not suffice to cover all elements between $a$ and $(1,0)$.
 Hence, for all $a\in\{0\}\times\N$ we have $a\cdot(1,0)=(1,0)$ and $M$ is not cancellative.
\end{proof}

We conclude that while locally finite auto-equivalent posets are group-embeddable (abelian) monoid posets, this does not hold any longer when dropping the local finiteness. We do not know whether auto-equivalent posets are (abelian) monoid posets in general. A first relaxation of the concept of auto-equivalent would be to drop the requirement that the submonoid $T<\End(P)$ has to be abelian. Free monoids yield such posets that are upwards oriented regular trees, see Figure~\ref{fig:bintree}. Adding relations one can obtain a big class of upwards oriented trees. What can be said about these posets?

\begin{figure}[h]
\centering
\includegraphics[height=.4\textwidth]{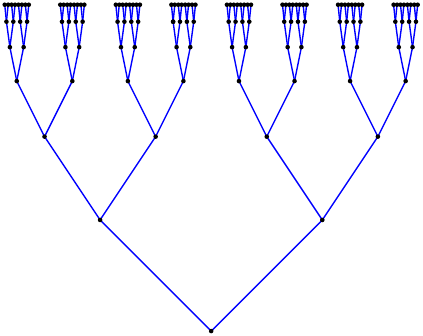}
\caption{The Cayley poset of the free monoid on two generators.}
\label{fig:bintree}
\end{figure}

Another relaxation is to allow posets without global minimum, but require that for all $x,y\in P$ there is a unique $\varphi_{xy}\in T$ such that $\varphi_{xy}(\upa x)=\upa y$ and the restriction of $\varphi_{xy}$ to $\upa x$ is an order isomorphism. 
More generally, as mentioned in the introduction, a poset is called \emph{uniform} if for every pair of elements $x,y$ the posets $\upa x$ and $\upa y$ are order isomorphic. This concept was introduced in~\cite{Mun-66} (with respect to principal downsets in semilattices). Clearly, auto-equivalent posets are uniform. In this most general setting we suspect:
\begin{conjecture}\label{conj:uniform}
 There are uniform posets that are not monoid posets.
\end{conjecture}
One strategy to construct such posets could be to take a transitive digraph $D=(V,A)$, that is not the Cayley graph of a monoid. Now, construct an infinite poset $P=(V\times \mathbb{Z}, \leq)$ such that $(u,i)\leq (v,j)$ if there is a directed walk of length $j-i$ from $u$ to $v$ in $D$. While transitive digraphs, that are not Cayley graphs of groups are known, see e.g.~\cite[Theorem 6]{Bre-94} or~\cite[Proposition 8.5]{Mar-85}, we end up with the following problem of independent interest:
\begin{question}\label{quest:graph}
 Is every transitive digraph the Cayley graph of a monoid?
\end{question}

\section{Conclusions}\label{sec:conclude}
We have explored the notion of Cayley posets by giving a natural inclusion diagram of different classes ad providing characterizations in terms of endomorphism monoids. In Theorem~\ref{thm:characterizationmono} we give a characterization of monoid posets. However, we wonder if the following strengthening could be true: 
If $P$ is a poset, $L<\End(P)$ a subsemigroup, and $F\subseteq P$ a principal upset, such that for every element $x\in P$ there is a unique $\varphi_x\in L$ such that $\varphi_x(F)=\upa x$, then $P$ is a monoid poset. By Theorem~\ref{thm:characterizationsemi} we know that in this case $P$ is a semigroup poset, but it is not clear if it is a monoid poset.

We have shown strictness of all inclusions in Figure~\ref{fig:diagram} and in particular, we provide small posets that are not Cayley. One first set of questions concerns the types of posets studied in Section~\ref{sec:firstseparations}. There is a natural way to define subposets of $P(\Z,\langle 2,3  \rangle)$ by choosing all elements between up to two maxima and up to two minima. We suspect that for large enough posets the properties of these posets depend only on the number of minima and maxima.

On the other hand, we give a rich set of constructions in order to obtain large classes of Cayley posets.  
A particular example are weak orders. While we know that they are full semigroup posets an open question is, whether bi-infinite weak orders are monoid posets only if all levels are of the same size (Conjecture~\ref{conj:weak})?

The last part of the paper is concerned with auto-equivalent posets. We show that if locally finite, such posets are (group-embeddable) monoid posets. We know that group-embeddability is lost when dropping local finiteness, but are these posets still monoid posets? At the end of Section~\ref{sec:autoequivalent} we discuss several further relaxations of this concept. Finally, one can ask, whether uniform posets, i.e., those where any two principal upsets are isomorphic, are monoid posets (Conjecture~\ref{conj:uniform}). A particularly intriguing question that comes up when studying the above conjecture, is whether thee are transitive digraphs that are not Cayley graphs of a monoid (Question~\ref{quest:graph}).

\subsubsection*{Acknowledgments.}
We thank Ulrich Knauer for his comments on the manuscript. The first author was partially supported by the Spanish \emph{Ministerio de Econom\'ia,
Industria y Competitividad} through grant MTM2016-78881-P. The second author was partially supported by the French \emph{Agence nationale de la recherche} through project ANR-17-CE40-0015 and by the Spanish \emph{Ministerio de Econom\'ia,
Industria y Competitividad} through grant RYC-2017-22701.

\bibliography{lit}
\bibliographystyle{my-siam}
\end{document}